\documentclass[11pt,oneside,reqno]{amsart}
\usepackage{amssymb, amscd, amsmath, amsthm, latexsym, hyperref,xcolor}
\usepackage{pb-diagram, fancyhdr, graphicx, psfrag}
\usepackage[text={6.3in,8.5in},centering,letterpaper,dvips]{geometry}
\usepackage{array,mathtools}
\usepackage{cite, accents}
\usepackage{gensymb}
\usepackage{verbatim, pbox}
\usepackage{graphicx}
\usepackage{wrapfig}
\usepackage{caption}

\graphicspath{{image_file/}}

\theoremstyle{plain}
\newtheorem{thm}{Theorem}[section] 

\theoremstyle{plain}
\newtheorem{lem}[thm]{Lemma}

\theoremstyle{plain}
\newtheorem{prop}[thm]{Proposition}

\theoremstyle{plain}
\newtheorem{cor}[thm]{Corollary}

\theoremstyle{plain}
\newtheorem{conj}[thm]{Conjecture}

\theoremstyle{remark}
\newtheorem*{rem}{Remark}

\theoremstyle{remark}

\theoremstyle{plain}
\newtheorem{question}{Question}

\theoremstyle{definition}
\newtheorem{defn}[thm]{Definition} 
\newtheorem{exmp}[thm]{Example} 

\def\cs{\mathop{\#}}

\newcommand{\spinc}{\ifmmode{{\mathfrak s}}\else{${\mathfrak s}$\ }\fi}
\newcommand{\spinct}{\ifmmode{{\mathfrak t}}\else{${\mathfrak t}$\ }\fi}
\newcommand{\spincw}{\ifmmode{{\mathfrak w}}\else{${\mathfrak w}$\ }\fi}

\def\R{\mathbb R}

\def\t{\textup{t}} 

\def \xra{\xrightarrow}

\def\U{\Upsilon}
\def\s{\mathfrak s}
\def\t{\mathfrak t}

\begin{document}


\title[Nonorientable surfaces bounded by knots]{Nonorientable surfaces bounded by knots: a geography problem}
\author{Samantha Allen}
\address{Samantha Allen: Department of Mathematics, Dartmouth College, Hanover, NH 03755 }
\email{samantha.g.allen@dartmouth.edu}

\begin{abstract}
The nonorientable 4--genus is an invariant of knots which has been studied by many authors, including Gilmer and Livingston, Batson, and Ozsv\'{a}th, Stipsicz, and Szab\'{o}.  Given a nonorientable surface $F \subset B^4$ with $\partial F = K\subset S^3$ a knot, an analysis of the existing methods for bounding and computing the nonorientable 4--genus reveals relationships between the first Betti number $\beta_1$ of $F$ and the normal Euler class $e$ of $F$.  This relationship yields a geography problem: given a knot $K$, what is the set of realizable pairs $(e(F), \beta_1(F))$ where $F\subset B^4$ is a nonorientable surface bounded by $K$?  We explore this problem for families of torus knots.  In addition, we use the Ozsv\'ath-Szab\'o $d$--invariant of two-fold branched covers to give finer information on the geography problem.  We present an infinite family of knots where this information provides an improvement upon the bound given by Ozsv\'ath, Stipsicz, and Szab\'o using the Upsilon invariant.
\end{abstract}

\maketitle
\section{Introduction}

One measure of the complexity of a knot in $S^3$ is the genus of the simplest surface that it bounds.  For instance, restricting to smooth, orientable surfaces in $B^4$ results in the 4--genus of the knot.  A similar variation is to consider smooth nonorientable surfaces in $B^4$ that are bounded by the knot---this yields the nonorientable 4--genus of the knot.
\begin{defn}
Let $F$ be a connected, nonorientable surface in $B^4$ with nonempty connected boundary $\partial F \subset S^3$.  Let $h(F)=\beta_1(F)= \text{ dim } H_1(F,\mathbb{Q})$ be the first Betti number of $F$ (also called the \textit{nonorientable genus of F}).  Then the \textit{nonorientable 4--genus of a knot K} is 
\begin{equation*}
    \gamma_4(K) = \min \left\{
    h(F) \middle| \begin{array}{c} \partial F = K \text{ and } F \text{ is a smoothly embedded,} \\ \text{nonorientable surface in } B^4\end{array}\right\}.
\end{equation*}
\end{defn}
\noindent Note that $\gamma_4(K)>0$, since any nonorientable surface with one boundary component has positive first Betti number.  (Some authors choose to let $\gamma_4(K) = 0$ in the case of a slice knot $K$.)

The nonorientable 4--genus is an invariant of knots which has been studied by many authors, including Viro \cite{viro}, Yasuhara \cite{yasuhara}, Gilmer and Livingston \cite{gilmer-livingston}, Batson \cite{batson}, and Ozsv\'{a}th, Stipsicz, and Szab\'{o} \cite{oss2}.  Each of these works offers bounds on $\gamma_4$ and obstructions for the existence of nonorientable surfaces in $B^4$ with boundary a given knot.

An analysis of the existing methods for bounding and computing the nonorientable 4--genus of a knot reveals that many of these methods depend on the Euler class $e$ of the normal bundle of the spanning surface, also called the \textit{normal Euler number} of the surface.  The author aims to further study this dependence.  We ask several questions, primarily:
\begin{question} \label{question 1}
Given a knot $K\subset S^3$, what is the set of realizable pairs \[(e(F), h(F)) = \text{(normal Euler number of a surface $F$, nonorientable genus of $F$)},\] where $F\subset B^4$ is bounded by $K$?
\end{question}
\noindent  For a given knot, we can plot these pairs in the $(e,h)$--plane and consider the region of realizable points.  This region (as in Figure \ref{wedge schematic}) is always a union of sets of the form 
$\left\{(e,h) \;:\; \left|a- \frac{e}{2}\right| \leq h\right\}$
for some value of $a$, ignoring some issues of parity.  This leads to the following question:
\begin{figure}[ht]
\includegraphics[width = 0.4\textwidth]{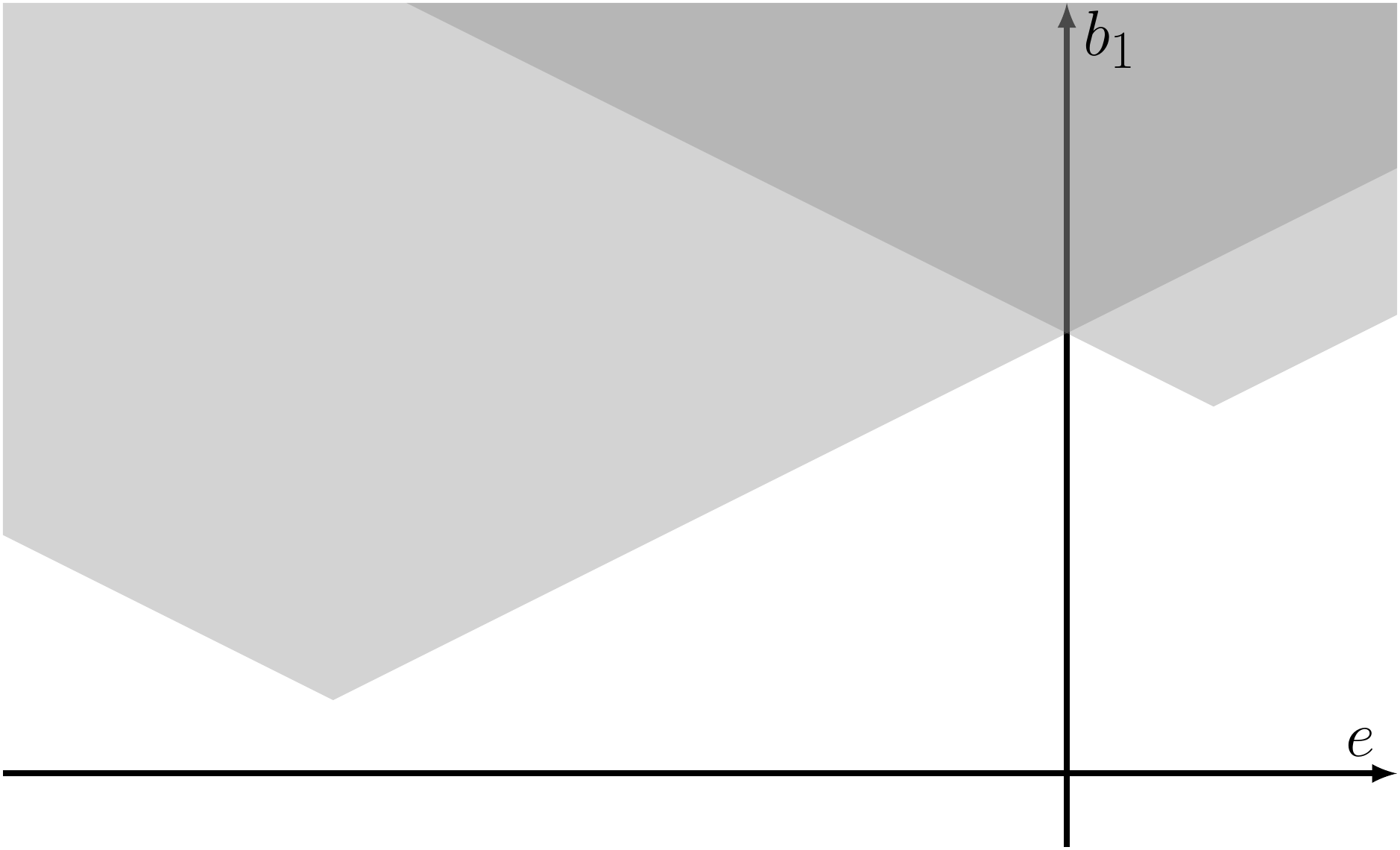}
\caption{}
\label{wedge schematic}
\end{figure}

\begin{question}
What shapes of regions are achievable?  How many global/local minima can occur?
\end{question}

For the remainder of this paper we will denote pairs $$\text{(normal Euler number of a surface $F$, nonorientable genus of $F$)}$$ by $(e,h)$.
We prove the following results.
\begin{thm}
For $T(2,n)$ with $n$ odd, the following pairs are realizable: $$(e,h) \in \{(-2n \pm 2m, 1+m+2l)\; |\; m,l\geq 0\} \cup \{(2+2m, n+m)\; |\; m\geq0\},$$
it is unknown if the following pairs are realizable:
        \begin{enumerate}
            \item if $n\equiv 1$ (mod 4),  $(e,h) = (4-2n+2m, 1+m)$ for $0\leq m\leq n-1$,
            \item if $n\equiv 3$ (mod 4),  $(e,h) = (8-2n+2m, 3+m)$ for $0\leq m\leq n-3$,
        \end{enumerate}
and all other pairs are not realizable.
\label{T(2,n)}
\end{thm}
\noindent Thus, for $T(2, 4k+1)$ and $T(2,4k+3)$ there are exactly $4k$ pairs for which realizability is unknown.

\begin{thm}
For $T(3,n)$ where $n=3k+d$ and $d = 1$ or $2$, the following pairs are realizable:  $$(e,h) = (-4n+2+4k\pm 2m, 1+m+2l) \text{ for } m,l\geq 0,$$
the following pairs are unknown:
        \begin{enumerate}
            \item if $n\equiv 1$ (mod 6),  $(e,h) = \left(\frac{8(1-n)}{3}+2+2m, 1+m\right)$ for $m\geq 0$,
            \item if $n\equiv 2$ (mod 6),  $(e,h) = \left(\frac{8(2-n)}{3}+2+2m, 3+m\right)$ for $m\geq 0$,
        \end{enumerate}
and all other pairs are not realizable.        
\label{T(3,n)}
\end{thm}
\noindent Note that when $n\equiv 4 \text{ or } 5$ (mod 6), the realizable pairs are completely determined.  In the other two cases there are infinitely many unknown points, all lying on a single line in the $(e,h)$--plane.

The following are conjectures of the author.
\begin{conj}
For $T(2,n)$ with $n$ odd, all unknown points are not realizable.
\end{conj}
\begin{conj}
For $T(3,n)$ with $n\equiv 1$ or $2$ (mod 6), all unknown points are not realizable.
\end{conj}
\begin{conj}
All torus knots have a single realizable ``minimal point".  In other words, for a torus knot $K$, there is exactly one realizable pair of the form $(e, \gamma_4(K))$,
\end{conj}

In addition, we give a family of knots for which the Ozsv\'{a}th-Szab\'{o} $d$--invariant can improve on the bound given by Ozsv\'{a}th, Stipsicz, and Szab\'{o} (using the $\Upsilon$ invariant). \\

\smallskip

\noindent{\it Acknowledgements}\ \  Thanks are due to Charles Livingston for guidance and careful reading of many early versions of this paper.  In addition, Ina Petkova provided many helpful comments and suggestions for improving the exposition.


\section{Background} 
\label{background}

We begin this section with a discussion of the nonorientable 4--genus of knots, followed by definitions and results concerning the normal Euler number.  Finally, we discuss the invariants and results that we use to prove the main theorems.
\subsection{The nonorientable 4--genus}
\label{nonori 4-g}

Computing the nonorientable 4--genus of knots is a difficult problem which remained relatively intractable until Heegaard Floer  theory entered the picture.  For at least a few families of knots it is simple to compute.  For example, the $(2, k)$--torus knot can be easily seen to have nonorientable 4--genus 1; see Figure \ref{2kand41}(a).  A single band move reveals that the $(3, k)$--torus knot also has nonorientable 4--genus 1; see Figure \ref{3kbandmove}.    Note that the $(2,k)$--torus knot actually bounds a surface in $S^3$, while the the $4$--ball is needed to realize a M\"obius band bounded by the $(3,k)$--torus knot.  In Figure \ref{2kand41}(b), we illustrate a nonorientable surface $F$ with boundary the figure-eight knot and $h(F) = 2$.  Viro \cite{viro} proved that the figure-eight knot has nonorientable genus greater than 1, and so $\gamma_4(4_1) = 2$.

\begin{figure}[ht]
\centering
\captionsetup{justification=centering}
\begin{tabular}{ccc}
\includegraphics[width = 0.2\textwidth]{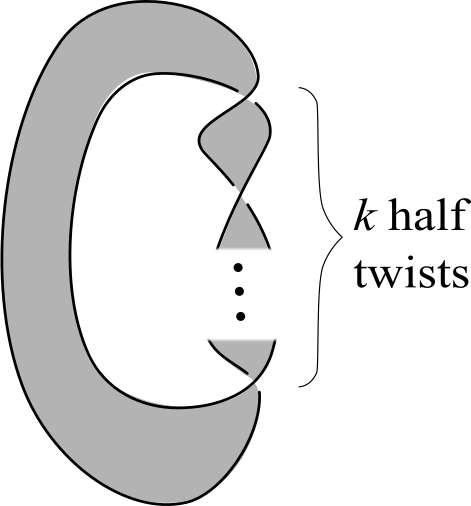} & $\;\;\;\;\;\;\;$ & \includegraphics[width = 0.2\textwidth]{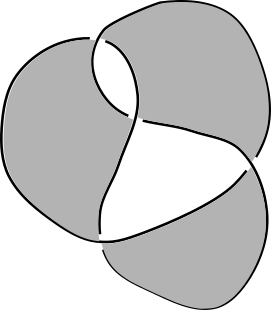}\\
(a) & $\;\;\;\;\;\;$ &(b)\\
\end{tabular}
\caption{(a) The $(2,k)$--torus knot bounding a M\"obius band.\\ (b) The figure-eight knot bounding a punctured Klein bottle.}
\label{2kand41}
\end{figure}

\begin{figure}[ht]
\centering
\captionsetup{justification=centering}
\begin{tabular}{ccccc}
$\vcenter{\hbox{\includegraphics[width = 0.2\textwidth]{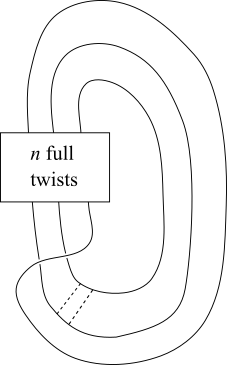}}}$ & $\xrightarrow{\text{Band move}}$ & $\vcenter{\hbox{\includegraphics[width = 0.2\textwidth]{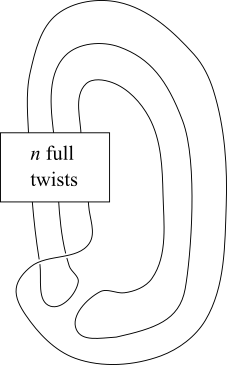}}}$& $\xra{\text{Isotopy}}$ & $\vcenter{\hbox{\includegraphics[width = 0.2\textwidth]{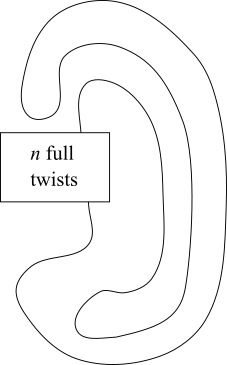}}}$\\
& & & & \\
$\vcenter{\hbox{\includegraphics[width = 0.2\textwidth]{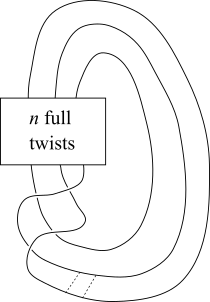}}}$ & $\xrightarrow{\text{Band move}}$ & $\vcenter{\hbox{\includegraphics[width = 0.2\textwidth]{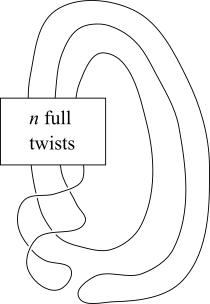}}}$& $\xra{\text{Isotopy}}$ & $\vcenter{\hbox{\includegraphics[width = 0.2\textwidth]{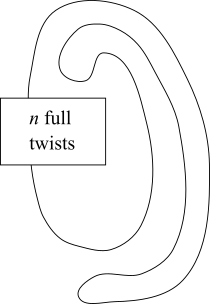}}}$\\

\end{tabular}
\caption{Performing an unoriented band move on $T(3,3n+1)$ (top) and \\ $T(3,3n+2)$ (bottom) results in the unknot. This implies that $\gamma_4(T(3,k))\leq 1$.}
\label{3kbandmove}
\end{figure}

An upper bound for $\gamma_4(K)$ is based on the $4$--genus, $g_4(K)$.  Let $G\subset B^4$ be a surface realizing $g_4(K)$.  Form the connected sum of $G$ with a real projective plane $P^2$ (away from the boundary of $G$); denote this new surface by $F$.  Then $F$ is a nonorientable surface in $B^4$ whose boundary is the knot $K$.  Thus $$\gamma_4(K) \leq h(F) = 2g_4(K) + 1.$$  This bound is sharp for all slice knots. Recently, Jabuka and Kelly \cite{jabuka-kelly} showed that this bound is sharp for some non-slice knots, as well: $\gamma_4(8_{18}) = 2g_4(8_{18})+1$.  However, since the $(2,k)$--torus knot has 4--genus $(k-1)/2$, the bound $2g_4(K)+1$ can be arbitrarily far from $\gamma_4(K)$.    Also, note that this bound implies that slice knots have nonorientable 4--genus equal to 1.

Some early progress towards lower bounds on the nonorientable 4--genus in the smooth case was made in 1975 when Viro \cite{viro} gave an obstruction to a knot bounding a M\"obius band in $B^4$ (using Witt classes of intersection forms of branched covers of the 4--ball branched over nonorientable surfaces).  In 1996, Yasuhara \cite{yasuhara} gave an obstruction to a knot bounding a M\"obius band in $B^4$ (using the knot signature and the Arf invariant).  In 2011, Gilmer and Livingston \cite{gilmer-livingston} gave an obstruction to a knot bounding a punctured Klein bottle in $B^4$ (again, using the knot signature and the Arf invariant).  Finally, in 2012, Batson \cite{batson} showed that the nonorientable 4--genus of a knot can be arbitrarily large by using the knot signature and the Heegaard-Floer $d$--invariant defined by Ozsv\'ath and Szab\'o.


\subsection{The normal Euler number}

In this section, we give a definition of the normal Euler number for a nonorientable surface in $B^4$ with boundary a knot.     We provide a summary here; see Gordon-Litherland \cite{gordon-litherland} for more details.

\begin{defn} \label{edef2}
Let $F \subset B^4$ be a nonorientable surface such that the boundary of $F$ is contained in $S^3 = \partial B^4$  and is a knot $K$.  The normal bundle $\nu (F)$ always admits a nowhere zero section $s$.  On the boundary, $s|_{\partial F}$ provides a framing of $K$.  Define the normal Euler number of the surface $F$ to be
$$e(F) := -\text{lk} (K,s(K)).$$
\end{defn}

   Gordon and Litherland \cite{gordon-litherland} give an algorithm for computing $e$ in the case where the nonorientable surface embeds in $\R^3$.  In some cases, we will build surfaces in $B^4$ using a sequence of band moves; these do not always embed in $S^3$.  For such surfaces, some care is required in the application of Definition \ref{edef2}.
\begin{exmp}\label{e from bands}
We build the M\"obius band in $B^4$ with boundary the trefoil $T(2,3)$ from a disk with a band added, and compute the normal Euler number.  See Figure \ref{einB4}.  Begin with a disk bounded by the unknot and take a 0--framed push-off of the unknot.  Add a band to the disk as shown to form the trefoil $T(2,3)$.  The knot $K=T(2,3)$ and the push-off $K'$ trace parallel surfaces in the 4--ball with intersection count $-$lk$(K, K') =-6$.  It follows that the knot $T(2,3)$ bounds a M\"obius band $F$ in $B^4$ with $e(F) = -6$ and so, for $T(2,3)$, the pair $(e,h) = (-6,1)$ is realizable.
\begin{figure}[ht]
\centering
\begin{tabular}{ccc}
\includegraphics[width = 0.2\textwidth]{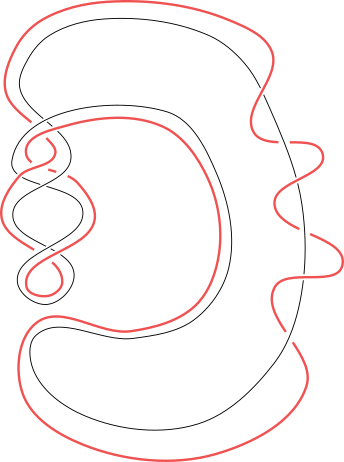} & \includegraphics[width = 0.2\textwidth]{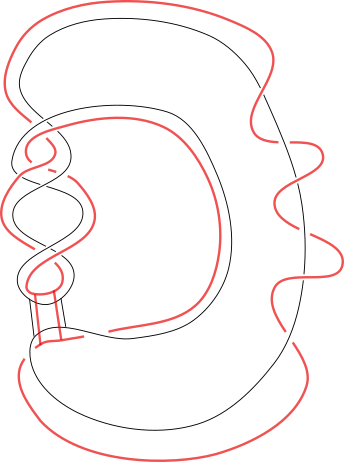} & \includegraphics[width = 0.2\textwidth]{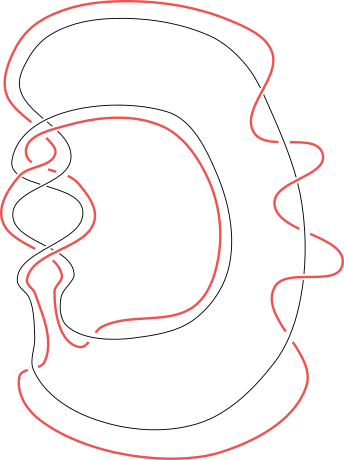}\\
\end{tabular}
\caption{}
\label{einB4}
\end{figure}
\end{exmp}

\subsection{Bounds and obstructions} Here we list some results for use later in the paper.

\subsubsection{The knot signature}
\label{sig}
The knot signature $\sigma(K)$ is a concordance invariant (see \cite{murasugi}).   Some useful properties are listed below.

\begin{thm}
For knots $K,J \subset S^3$,
\begin{enumerate}
\item $\sigma (K\cs  J) = \sigma (K)+\sigma (J)$.
\item $\sigma (-K)=-\sigma (K)$.
\item If $K$ is slice, $\sigma (K)=0.$
\item $\sigma (K)$ is even for all $K$.
\end{enumerate}
\end{thm}
In 1978, Gordon and Litherland proved the following relationship:

\begin{thm}[Gordon-Litherland, \cite{gordon-litherland}]\label{sign(W(F))}
Let $K\subset S^3$ be a knot that bounds a connected surface $F\subset B^4$ and let $\Sigma(F)$ be the two-fold cover of $B^4$ branched over F.  Then \[\sigma(K) = \text{sign}(\Sigma(F))+\frac{1}{2}e(F),\]
where, sign$(\Sigma (F))$ denotes the signature of the intersection form $Q_{\Sigma (F)}$ on $H_2(\Sigma (F))$.
\end{thm}
\noindent This has the following corollary.
\begin{cor}\label{sigma wedge}
Let $K\subset S^3$ be a knot and let $F\subset B^4$ be a nonorientable surface such that $\partial F=K$.  Then \[\left|\sigma(K)-\frac{e(F)}{2}\right|\leq h(F).\]
\end{cor}
\noindent The corollary will follow easily from the theorem and the following lemma of Massey:
\begin{lem}[Massey~\cite{massey}] \label{massey lem}
Let $F\subset B^4$ be a connected surface bounded by a knot $K\subset S^3$ and let $\Sigma (F)$ be the two-fold cover of $B^4$ branched over $F$.  Then $b_2(\Sigma (F)) = b_1(F)$. 
\end{lem}

\begin{proof}[Proof of Corollary \ref{sigma wedge}]
Let $K\subset S^3$ and let $F\subset B^4$ be a nonorientable surface such that $\partial F=K$.  Then 
\[\left|\sigma(K) - \frac{1}{2}e(F)\right|= |\text{sign}(\Sigma(F))| \leq b_2(\Sigma (F)) = b_1(F),\]
where the first and last equalities follow from Theorem \ref{sign(W(F))} and Lemma \ref{massey lem}, respectively.
\end{proof}

\subsubsection{The Arf invariant}  The Arf invariant Arf$(K)$ is an invariant of knots satisfying the following:
$$\text{Arf}(K) = \left\{ \begin{array}{ll} 0 & \text{if } \Delta_K(-1) \equiv \pm 1 \text{ (mod } 8) \\
									  1 & \text{if } \Delta_K(-1) \equiv \pm 3 \text{ (mod } 8) . \end{array} \right.$$
In 2010, Gilmer and Livingston gave the following obstruction to a knot bounding a Klein bottle.
\begin{thm}[Gilmer-Livingston, \cite{gilmer-livingston}] \label{klein obstruction}
If $K$ bounds a punctured Klein bottle $F$ in $B^4$ and $\Sigma (F)$, the two-fold cover of $B^4$ branched over $F$, has a positive definite intersection form, then \[\sigma (K)+4\text{\normalfont Arf}(K)\equiv 0,2, \text{ or }\; 4 \;\text{\normalfont  (mod 8)}.\]  If $\Sigma (F)$ is negative definite, then \[\sigma (K)+4\text{\normalfont Arf}(K)\equiv 0,4, \text{ or } \; 6 \; \text{\normalfont (mod 8)}.\]
\end{thm} 

\subsubsection{The Upsilon invariant} The  Ozsv\'ath-Stipsicz-Szab\'o Upsilon invariant is another concordance invariant which yields a relationship between the normal Euler number and nonorientable genus.  The definition of the Upsilon invariant $\U_{K}(t)$ arises from the Heegaard Floer knot complex.  It is a piecewise linear function on $[0,2]$.  Some useful properties are listed below.
\begin{thm}[As in \cite{oss1}]
For knots $K,J \subset S^3$,
\begin{enumerate}
\item $\U_{K\cs  J}(t) = \U_{K}(t)+\U_{J}(t)$.
\item $\U_{-K}(t)=-\U_{K}(t)$.
\item If $K$ is slice, $\U_K(t)=0.$
\item If $K$ is an alternating knot, $\U_K(1) = \sigma(K)/2$.
\end{enumerate}
\end{thm}
In 2015, Ozsv\'{a}th, Stipsicz, and Szab\'{o} proved the following.

\begin{thm}[Ozsv\'{a}th-Stipsicz-Szab\'{o}, \cite{oss2}]
Suppose that $F\subset [0,1]\times S^3$ is a (not necessarily orientable) smooth cobordism from the knot $K_0 \subset \{0\}\times S^3$ to the knot $K_1 \subset \{1\}\times S^3$.  Then
\begin{equation}
    \left|\U _{K_0}(1)-\U _{K_1}(1)+\frac{e(F)}{4}\right| \leq \frac{h(F)}{2}.
\end{equation}
\label{OSS ineq}
\end{thm}
\noindent To apply this, consider a cobordism $F$ from a knot $K$ to the unknot.  Capping off the unknot with a disk, we get a surface $F'$ in $B^4$ that bounds $K$ and has $e(F') = e(F)$ and $h(F') = h(F)$.  Since Upsilon is identically zero for the unknot, we have the following:
\begin{cor}\label{OSS cor}
Let $K\subset S^3$ be a knot and let $F\subset B^4$ be a nonorientable surface such that $\partial F=K$.  Then \[\left|-2\U_K(1)+\frac{e(F)}{2}\right|\leq h(F).\]
\end{cor}
\begin{rem} For an alternating knot $K$, $\sigma(K)/2 = \Upsilon_K(1)$ and so this bound is equivalent to that of Gordon and Litherland.\end{rem}
\noindent Combining this with Corollary \ref{sigma wedge}, Ozsv\'{a}th, Stipsicz, and Szab\'{o} gave the following lower bound on the nonorientable 4--genus.
\begin{cor}[\cite{oss2}]\label{OSS bound}
Let $K\subset S^3$ be a knot.  Then \[\left|\U_K(1)-\frac{\sigma(K)}{2}\right|\leq \gamma _4(K).\]
\end{cor}

Recently, Jabuka and Van Cott \cite{jabuka-van-cott} addressed a conjecture of Batson \cite{batson} by using Corollary \ref{OSS bound} to compute the nonorientable 4--genus of many families of torus knots.


\section{The geography problem for families of torus knots}
\label{evsh}
\subsection{A detailed example}
Here we return to Question \ref{question 1} and begin with an example: the trefoil knot.  As we saw in Section \ref{nonori 4-g}, the trefoil knot $T(2,3)$ bounds a M\"obius band with normal Euler number $-6$, so $\gamma_4(T(2,3)) = 1$ and the pair $(e,h) = (-6,1)$ is realizable for $T(2,3)$.  We would like to identify all pairs $(e, h)$ that arise for nonorientable surfaces $F$ whose boundary is $T(2,3)$.  

To start, we look at surfaces that can be realized by modifying surfaces we already know are realizable.  In particular, given a nonorientable surface $F\subset B^4$ whose boundary is a knot $K\subset S^3$, we can form the connected sum of $F$ with the real projective plane $P^2$ (away from the boundary) to form a new nonorientable surface $F'$ with boundary $K$.  This surface will have $e(F') = e(F) \pm 2$ (since $P^2$ has Euler number $\pm 2$) and $h(F') = h(F)+1$.

Thus, since $(-6, 1)$ is realizable for the trefoil, so are $(-8, 2)$ and $(-4, 2)$.  Adding another $P^2$ shows that $(-10,3)$, $(-6, 3)$, and $(-2, 3)$ are realizable as well.  We can continue this process indefinitely to get the following proposition.

\begin{prop} \label{trefoil init}
For $T(2,3)$, all pairs of the form 
$$(e,h) = (-6 \pm 2n, 1+n+2m)$$
 for $m,n\geq 0$ are realizable.  
\end{prop}

 \noindent Figure \ref{trefoil grid 1} gives a visual representation of the proposition.  In the figure, many lattice points are omitted.  These are points that are ruled out by a theorem of Massey \cite{massey}: for all nonorientable surfaces $F$ with boundary a knot, \begin{equation} \label{ehparity}e(F) \equiv 2h(F) \text{ (mod } 4).\end{equation}

\begin{figure}[ht]
  \centering
  \captionsetup{justification = centering}
  \includegraphics[width = 0.4\textwidth]{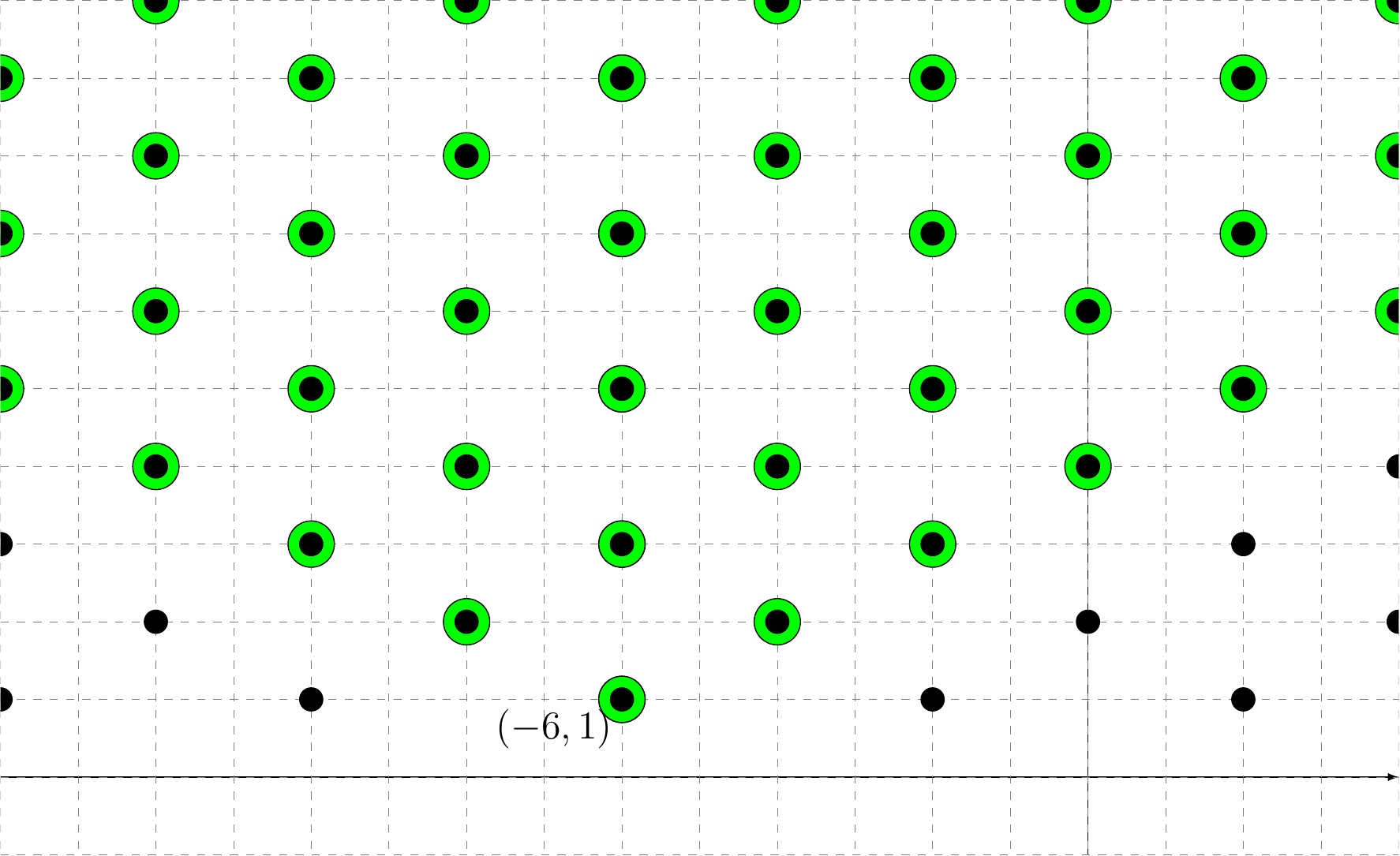}
  \caption{$(e,h)$ pairs for $T(2,3)$ based on Proposition \ref{trefoil init}. \\ Highlighted points indicate realizable pairs.}
  \label{trefoil grid 1}
\end{figure}

 By noticing that a torus knot $T(2, k)$ with $k$ odd bounds a M\"obius band with normal Euler number $-2k$, we get an initial realizable pair of $(e,h) = (-2k, 1)$ and the result above generalizes.
\begin{cor}
For $T(2,k)$ with $k$ odd, all pairs of the form $(e,h) = (-2k \pm 2n, 1+n+2m)$ for $m,n\geq 0$ are realizable.
\end{cor}

Another strategy for constructing nonorientable surfaces with boundary a knot is to use the 4--genus of the knot.  For a given knot $K$, we can find an orientable surface $S \subset B^4$ with $\partial S=K$ and genus $g=g_4(K)$.  The surface $S$ is then a punctured connected sum of $g$ tori.  If we form a connected sum of $S$ with a projective plane $P^2$, then we get a nonorientable surface $S'$ that has $h(S') = 2g+1$.  Since $S$ is orientable, it contributes 0 to the normal Euler number.  Thus $e(S')=\pm2$.  This yields a point on our $(e,h)$ graph for $K$.  For the torus knot $T(2,k)$, we know that $g_4(T(2,k))=(k-1)/2$ and so this yields the points $(\pm 2,k)$.  For the trefoil knot, (2,3) is a new point.  By forming the connected sum of $S'$ with $n$ copies of $P^2$, we see that the pairs $(2+2n, 3+n)$ are also new realizable pairs for $T(2,3)$.  See Figure \ref{trefoil grid 3} for the updated graph.

\begin{figure}[ht]
  \centering
  \captionsetup{justification = centering}
  \includegraphics[width = 0.4\textwidth]{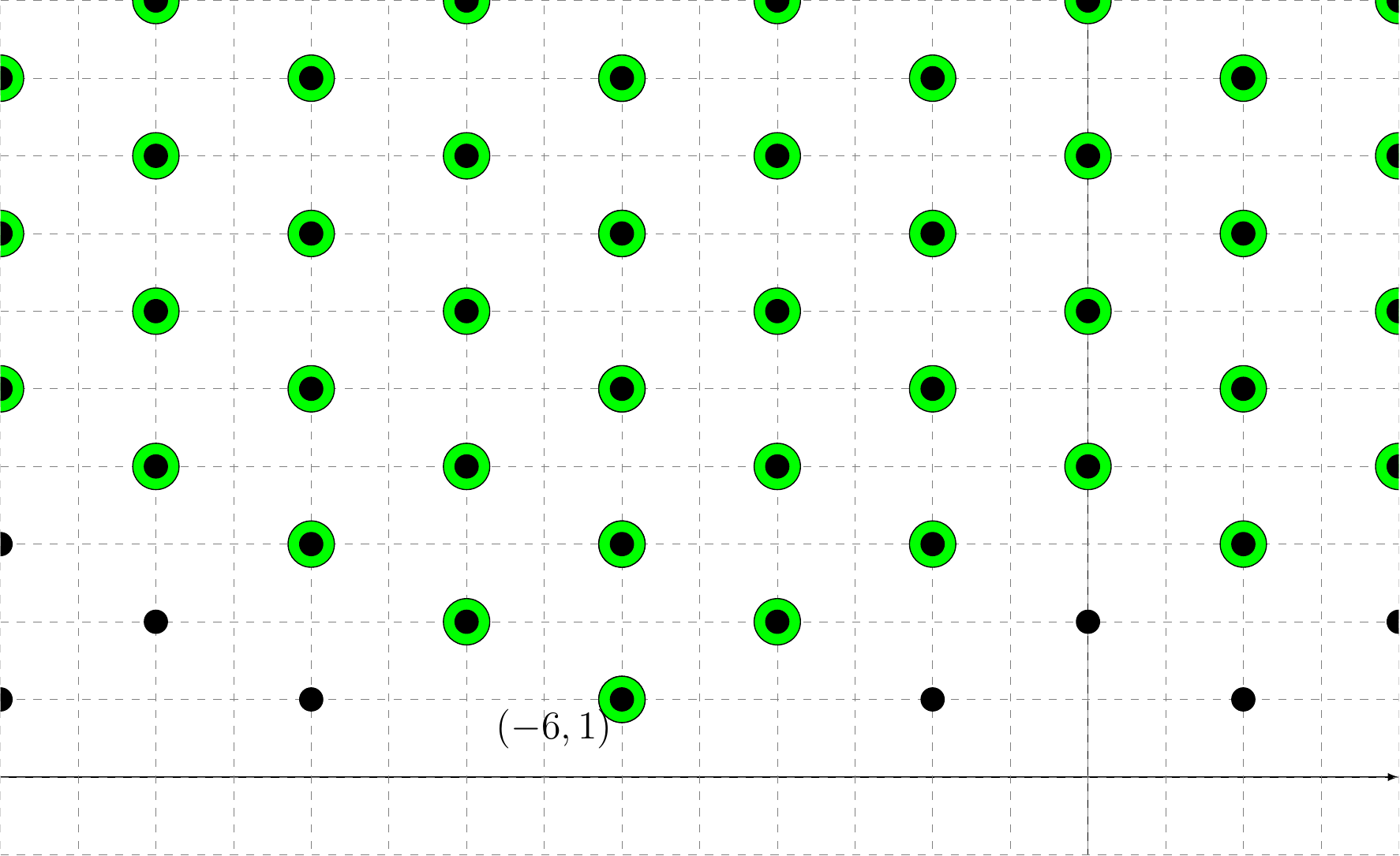}
  \caption{Updated $(e,h)$ pairs for $T(2,3)$. \\ Highlighted points indicate realizable pairs}
  \label{trefoil grid 3}
\end{figure}

Next, we apply Corollaries \ref{sigma wedge} and \ref{OSS cor}:  if $F\subset B^4$ is a nonorientable surface with boundary equal to a knot $K\subset S^3$,then
$$\left|\sigma(K)-\frac{e(F)}{2}\right|\leq h(F) \;\text{ and }\; \left|2\U_K(1)-\frac{e(F)}{2}\right|\leq h(F).$$
 Notice that since $T(2,3)$ is an alternating knot, Corollaries \ref{sigma wedge} and \ref{OSS cor} yield the same inequality.  Applying Corollary \ref{sigma wedge} to the knot $T(2,3)$ and computing that $\sigma(T(2,3)) = -2$, we can restrict our search for realizable $(e,h)$ pairs to those satisfying $\left|-2-\frac{e}{2}\right|\leq h$.  Factoring in our previous conclusions, we see that there are only two unknown pairs, $(-2, 1)$ and $(0, 2)$,  as shown in Figure \ref{trefoil grid 2}.

To rule out the final unknown points, we use the obstruction for punctured Klein bottles in Theorem \ref{klein obstruction}.
For the knot $T(2,3)$, we first apply Theorem \ref{sign(W(F))} to see that both unknown points $(e,h)$ in Figure \ref{trefoil grid 2} must correspond to a surface $F$ with sign$(\Sigma (F)) = -h$.  Lemma \ref{massey lem} tells us that $b_2(\Sigma (F))=b_1(F) = h$.  Thus, both of the unknown points correspond to surfaces $F$ for which $\Sigma(F)$ has negative definite intersection form.
Since the pair $(e, h)=(0,2)$ corresponds to a surface $F$ such that $\Sigma (F)$ is negative definite, we apply Theorem \ref{klein obstruction}.  As $\sigma (T(2,3)) = -2$ and $\text{\normalfont Arf}(T(2,3)) = 1$, Theorem \ref{klein obstruction} rules out the realization of this pair.  As a consequence, the pair $(e,h) = (-2,1)$ cannot be realized.  This fully determines the realizable $(e,h)$ pairs for $T(2,3)$.

\begin{figure}[ht]
  \centering
  \captionsetup{justification = centering}
  \includegraphics[width = 0.4\textwidth]{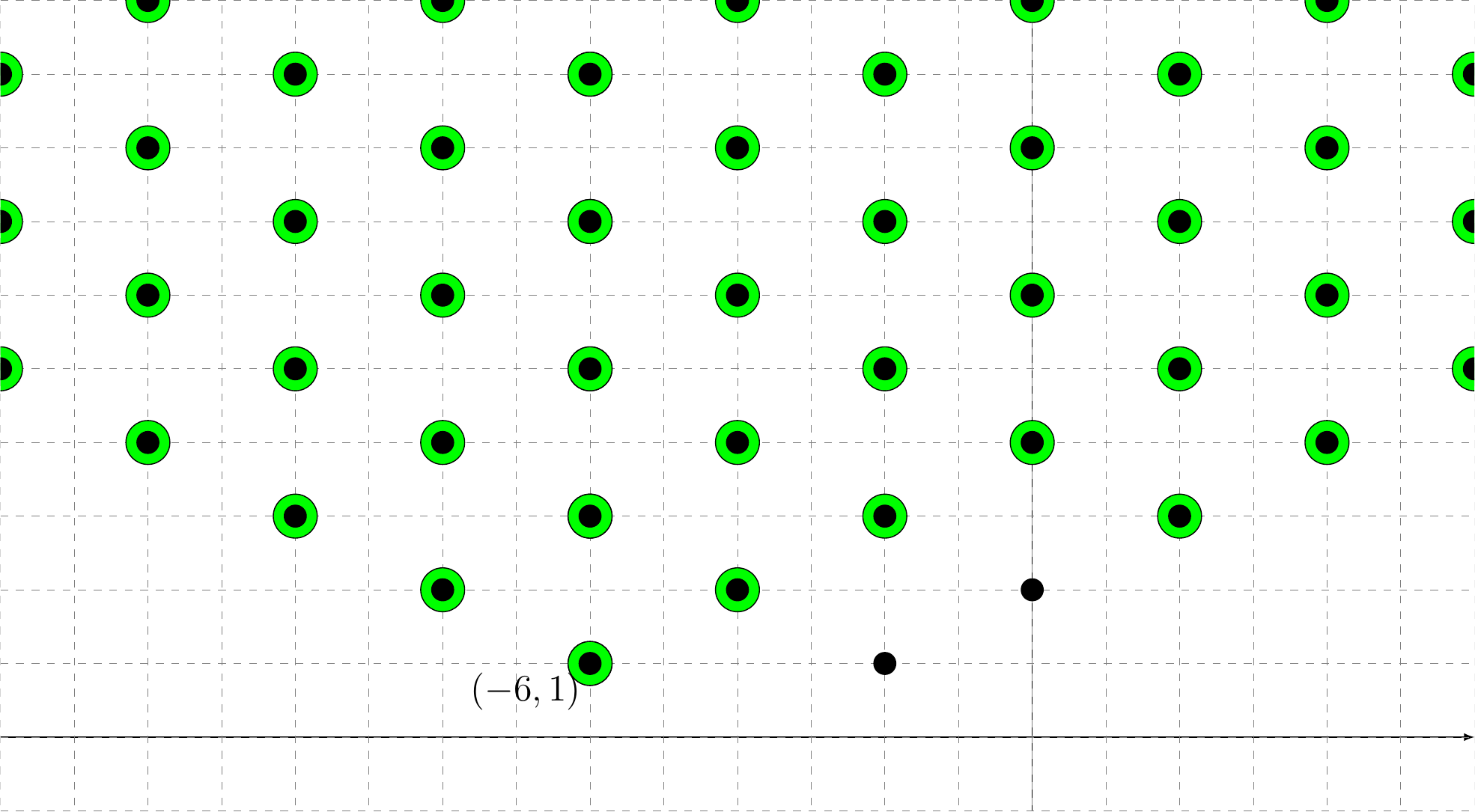}
  \caption{Updated $(e,h)$ pairs for $T(2,3)$ after application of Corollary \ref{sigma wedge}. \\ Highlighted points indicate realizable pairs}
  \label{trefoil grid 2}
\end{figure}
\section{Proof of Theorem \ref{T(2,n)}}

Now, we move on to the families of $(2,n)$ and $(3,n)$ torus knots.  We will henceforth ignore points that have been ruled out by Equation \ref{ehparity}. 
\begin{proof}[Proof of Theorem \ref{T(2,n)}] 
The knot $T(2,n)$ bounds a M\"obius band with normal Euler number $-2n$; see Figure \ref{2kand41}(a).  This yields the realizable pair $(-2n,1)$.  Since $g_4(T(2,n)) = \frac{n-1}{2}$, we can form a connected sum with $P^2$ to get nonorientable surfaces $S^{\pm}$ with boundary $T(2,n)$ such that $e(S^{\pm}) = \pm 2$ and $h(S^{\pm}) = 2\cdot \left(\frac{n-1}{2}\right) +1 = n$.  So the points $(\pm 2, n)$ are realizable for $T(2,n)$. Finally, by forming connected sums with $P^2$, we get wedges $W_1, W_2, W_3$ of realizable pairs starting at all three of these initial realizable points:
$$W_1 = \left\{ (e,h) : \left|-n-\frac{e}{2}\right|+1\leq h\right\},$$
$$W_2 = \left\{ (e,h) : \left|-1-\frac{e}{2}\right|+n\leq h\right\},$$
and
$$W_3 = \left\{ (e,h) : \left|1-\frac{e}{2}\right|+n\leq h\right\}.$$
See Figure \ref{T(2,n) grid} for a schematic picture of the $(e,h)$--graph of the regions.  Notice that $W_2\subset W_1$.

\begin{figure}[ht]
  \centering
  \captionsetup{justification = centering}
  \includegraphics[width = 0.4\textwidth]{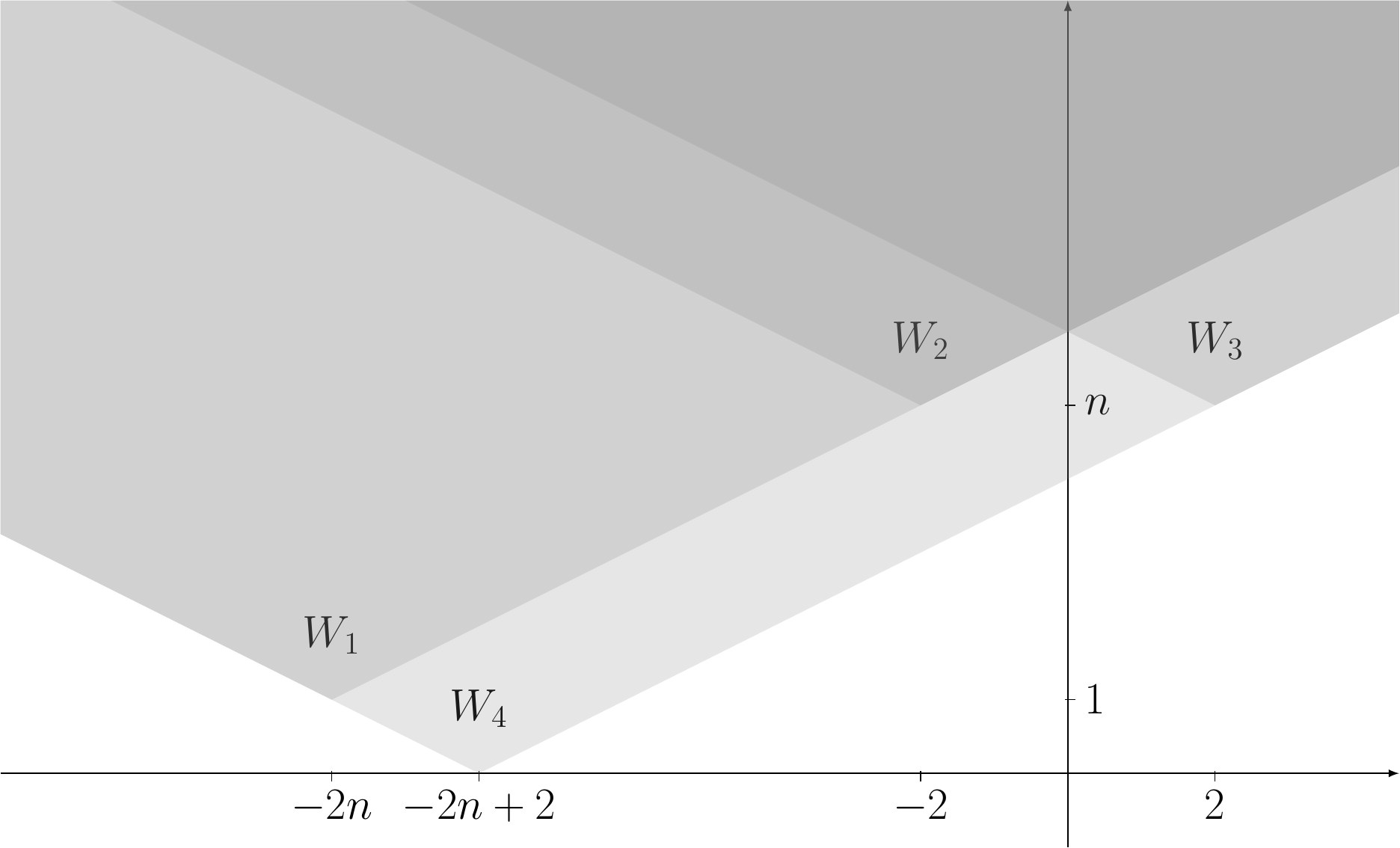}
  \caption{Schematic picture of the $(e,h)$--graph of the regions $W_1, W_2, W_3$ and $W_4$ for the knot $T(2,n)$.}
  \label{T(2,n) grid}
\end{figure}

We have that $\sigma (T(2,n))=-(n-1)$. Thus, Theorem \ref{sigma wedge} rules out all pairs outside of the region 
$$W_4 = \left\{ (e,h) : \left|-(n-1)-\frac{e}{2}\right|\leq h\right\}.$$
In other words, all unknown points lie in $U := (W_1\cup W_3)^c \cap W_4$; there are exactly $n-1$ unknown points and all of the unknown points lie on the line $h=\frac{e}{2}+(n-1)$.

Finally, we attempt to rule out pairs by applying Theorem \ref{klein obstruction}.  First, from Theorem \ref{sign(W(F))}, we see that for a surface $F$ given by a point $(e,h)=\left(e, \frac{e}{2}+(n-1)\right) \in U$,  Sign$(\Sigma (F)) = -h$.  Therefore, Lemma \ref{massey lem} reveals that all of the unknown points correspond to surfaces $F$ with $\Sigma (F)$ negative definite.  Consider the point $(-2n+6,2)$ corresponding to a punctured Klein bottle.  

Since $$\Delta_{T(2,n)}(t) = 1-t+\cdots -t^{n-2}+t^{n-1},$$ $\textrm{Arf}(T(2,n)) = 0$ if and only if $n\equiv \pm 1 \text{ (mod 8)}$. Consider the following table of computations:
\begin{center}
\begin{tabular}{c|c|c|c|c}
$n$ (mod 8) &$\sigma$ & Arf & $\sigma +4$Arf & (mod 8)\\
\hline
1&$-(n-1)$&0&$-n+1$&0\\
3&$-(n-1)$&1&$-n+5$&2\\
5&$-(n-1)$&1&$-n+5$&0\\
7&$-(n-1)$&0&$-n+1$&2\\
\end{tabular}
\end{center}
While Theorem \ref{klein obstruction} provides no obstruction for $n\equiv 1$ (mod 4), we can rule out the pair $(e,h)=(6-2n, 2)$ (and so also $(4-2n,1)$) for $n\equiv 3$ (mod 4).  This proves Theorem \ref{T(2,n)}.
\end{proof}


\section{Proof of Theorem \ref{T(3,n)}}
Before proving Theorem \ref{T(3,n)}, we compute the value of the knot signature and of the Upsilon invariant at $t=1$ for the family of torus knots $T(3,n)$.

\begin{lem} \label{signature}  For the torus knot $T(3,n)$,
\begin{equation*}
    \sigma(T(3,n)) = \left\{\begin{array}{ll}
	   \frac{-4n+4}{3}&\text{ if } n\equiv 1 \text{ (mod 6)}\\
        \frac{-4n+2}{3}&\text{ if } n\equiv 2 \text{ (mod 6)}\\
        \frac{-4n-2}{3}&\text{ if } n\equiv 4 \text{ (mod 6)}\\
        \frac{-4n-4}{3}&\text{ if } n\equiv 5 \text{ (mod 6)}.\\
        \end{array} \right.
\end{equation*}
\end{lem}

\begin{proof}
We use the recursive formulas of Gordon, Litherland, and Murasugi \cite{gordon-litherland-murasugi}.  Consider $T(3,n)$ where $n=6k+d$ with $k\geq 1$ and $d\in\{1, 2, 4, 5\}$.  Then 
$$\sigma(T(3,n)) = \sigma(T(3, 6(k-1)+d))-8 = \sigma(T(3, 6(k-2)+d))-16 = \cdots = \sigma(T(3,d))-8k.$$
So, since $\sigma(T(3,1)) = 0, \;\sigma(T(3,2)) = -2,\; \sigma(T(3,4)) = -6, \text{ and } \sigma(T(3,5)) = -8,$ we have
$$\sigma(T(3, 6k+1)) = -8k,\; \sigma(3, 6k+2) = -2-8k,$$ $$ \sigma(3, 6k+3) = -6-8k, \text{ and } \sigma(T(3, 6k+5)) = -8-8k.$$
Substituting $k = (n-d)/6$ proves the lemma.
\end{proof}
\begin{lem} \label{upsilon}
For the torus knot $T(3,n)$,
\begin{equation*}
    \U_{T(3,n)}(1) = \left\{\begin{array}{ll}
	   \frac{-2n+2}{3}&\text{ if } n\equiv 1 \text{ (mod 3)}\\
        \frac{-2n+1}{3}&\text{ if } n\equiv 2 \text{ (mod 3)}.\\
        \end{array} \right.
\end{equation*}
\end{lem}
To prove the lemma, we use the following recursive formula given by Feller and Krcatovich \cite{feller-krcatovich}.
\begin{thm}[ \cite{feller-krcatovich}] \label{upsilon recursive}
Let $a<b$ be two coprime positive integers.  Then $$\U_{T(a,b)}(t) = \U_{T(a, b-a)}(t) + \U_{T(a, a+1)}(t).$$
\end{thm}
\begin{proof}[Proof of Lemma \ref{upsilon}]
 Consider $T(3,n)$ where $n=3k+d$ with $k\geq 1$ and $d=1$ or $2$.  Then, by Theorem \ref{upsilon recursive}, 
 $$\U_{T(3,n)}(t) = \U_{T(3, 3(k-1)+d)}(t) + \U_{T(3,4)}(t)= \U_{T(3, 3(k-2)+d)}(t) + 2\U_{T(3,4)}(t)$$ $$=\cdots = \U_{T(3, d)}(t) + k\U_{T(3,4)}(t).$$
 Since $\U_{T(3,1)}(1) = 0$, $\U_{T(3,2)}(1) = -1$, and $\U_{T(3,4)}(1) = -2$, we have
 $$\U_{T(3, 3k+1)}(1) = -2k \text{ and } \U_{T(3, 3k+2)}(1) = -1-2k.$$
 Substituting $k = (n-d)/3$ proves the lemma.
\end{proof}

\begin{proof}[Proof of Theorem \ref{T(3,n)}]
In Section \ref{nonori 4-g}, we showed that $\gamma_4(T(3,n))=1$.  We compute the Euler number for this M\"obius band using the method in Example \ref{e from bands}.  For the torus knot $T(3, 3k+1)$, begin with a disk bounded by the unknot and take a $0$-framed push-off of the unknot, as in Figure \ref{1(mod 3)}(a).  Add a band to the disk as shown in Figure \ref{1(mod 3)}(b).  The resulting knot $K$ and the push-off $K'$ (shown as a dotted line in Figure \ref{1(mod 3)}(c)) trace parallel surfaces in $B^4$ with intersection count $-\text{lk}(K, K') = \frac{-8n+2}{3}$.  Thus, $T(3, 3k+1)$ bounds a M\"obius band in $B^4$ with normal Euler number $\frac{-8n+2}{3}$.  The same process shows that $T(3, 3k+2)$ bounds a M\"obius band in $B^4$ with normal Euler number $\frac{-8n-2}{3}$.

\begin{figure}[ht]
\centering
\begin{tabular}{ccc}
 \hspace{1cm}\includegraphics[width = 0.2\textwidth]{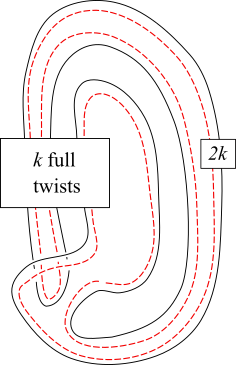} \hspace{1cm} & \includegraphics[width = 0.2\textwidth]{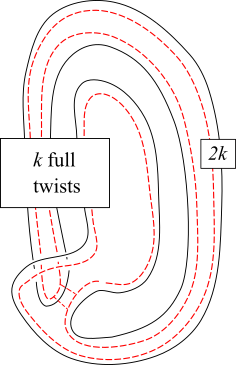} \hspace{1cm} & \includegraphics[width = 0.2\textwidth]{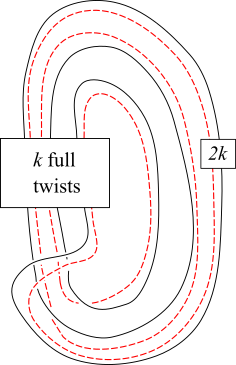} \hspace{1cm}\\
(a) & (b) & (c) \\
0-framed unknot & band move \hfill & $T(3,3k+1)$ with push-off\\
\end{tabular}
\caption{}
\label{1(mod 3)}
\end{figure}

 Let $e_0$ denote $e(F)$ for the surface described above. In other words, let
$$e_0 = \left\{\begin{array}{ll}
	\frac{-8n+2}{3} & n\equiv 1 \text{ (mod 3)}\\
	\frac{-8n-2}{3} & n\equiv 2 \text{ (mod 3)}.\\
\end{array}\right.$$ 
Then, in the $(e,h)$--graph for $T(3,n)$, this yields an initial $(e,h)$ pair of $(e_0,1)$.  Since $g_4(T(3,n)) = n-1$, there exist nonorientable surfaces with boundary $T(3,n)$ and $(e,h)$ equal to $(\pm 2, 2n-1)$.  By forming connected sums with $P^2$, we get wedges $W_1, W_2, W_3$ of realizable pairs starting at all three of these initial realizable points:
$$W_1 = \left\{ (e,h) : \left|\frac{e_0}{2}-\frac{e}{2}\right|+1\leq h\right\},$$
$$W_2 = \left\{ (e,h) : \left|-1-\frac{e}{2}\right|+2n-1\leq h\right\},$$
and
$$W_3 = \left\{ (e,h) : \left|1-\frac{e}{2}\right|+2n-1\leq h\right\}.$$
Notice that $W_2, W_3\subset W_1$.  See Figure \ref{T(3,n) grid 1} for schematic pictures of the $(e,h)$--graphs of the regions.

\begin{figure}[ht]
  \centering
  \captionsetup{justification = centering}
  \includegraphics[width = 0.4\textwidth]{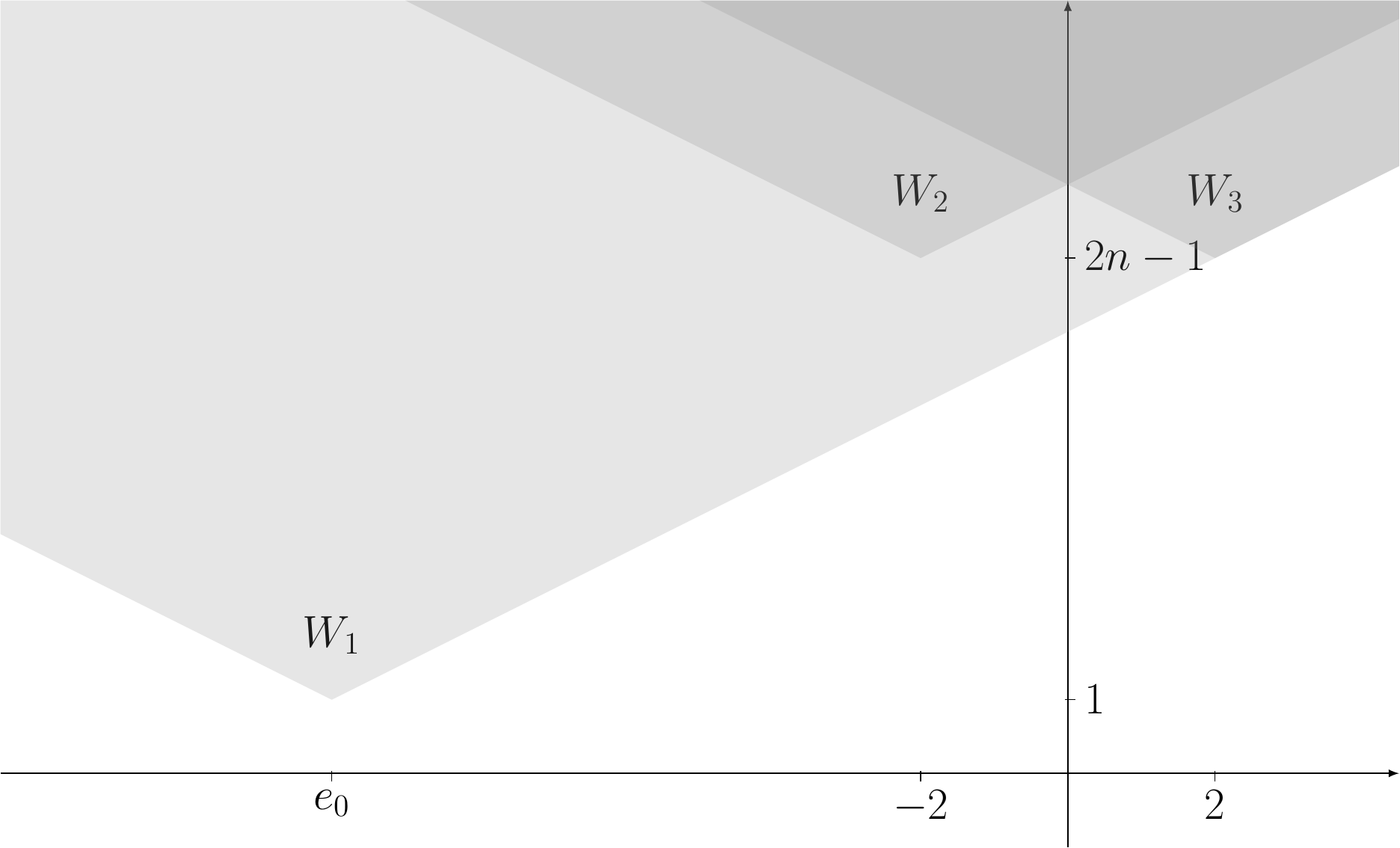}
  \caption{Schematic picture of the $(e,h)$--graph of the regions $W_1, W_2$ and $W_3$ for the knot $T(3,n)$.}
  \label{T(3,n) grid 1}
\end{figure}

Next, we determine which $(e,h)$ pairs can be ruled out using the knot signature and the Upsilon invariant.  We apply Lemmas \ref{signature} and \ref{upsilon} to compute that
\begin{equation*}
    \sigma(T(3,n)) - 2\U_{T(3,n)}(1) = \left\{\begin{array}{ll}
	   0&\text{ if } n\equiv 1  \text{ or 2 (mod 6)}\\
        -2&\text{ if } n\equiv 4 \text{ or 5 (mod 6).}\\
        \end{array} \right.
\end{equation*}

The combination of Corollaries \ref{sigma wedge} and \ref{OSS cor} implies that all realizable pairs lie in the intersection of the regions
\[R_1:=\left\{(e,h): \left|\sigma(K)-\frac{e}{2}\right|\leq h \right\} \text{ and } R_2:=\left\{(e,h): \left|-2\U_K(1)+\frac{e}{2}\right|\leq h\right\}.\]
See Figures \ref{T(3,n) grid 2b} and \ref{T(3,n) grid 2a} for the schematic pictures of the $(e,h)$--graphs for the knot $T(3,n)$ including regions $R_1$ and $R_2$. 
\begin{figure}[ht]
  \centering
  \captionsetup{justification = centering}
  \includegraphics[width = 0.4\textwidth]{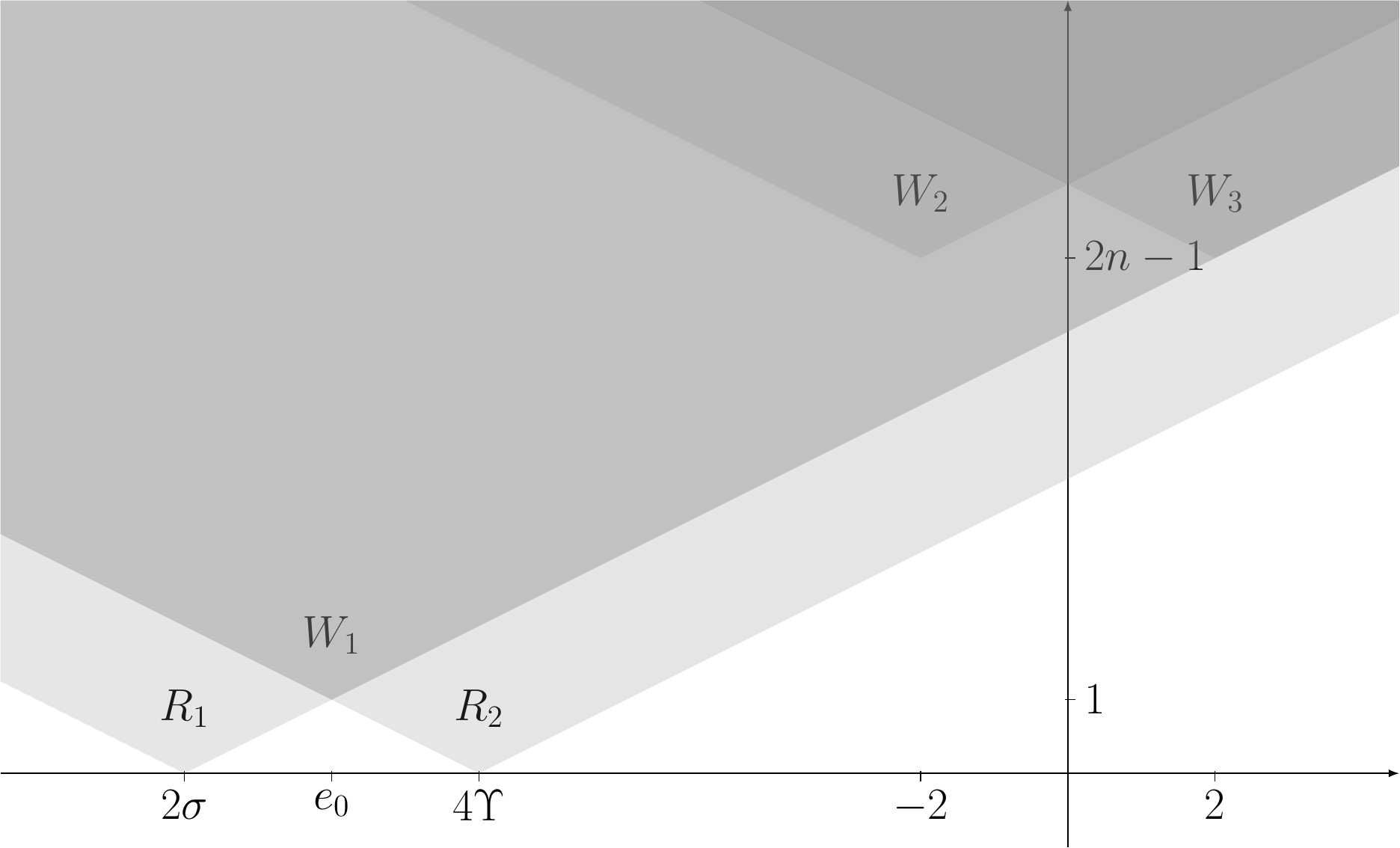}
  \caption{Schematic picture of the $(e,h)$--graph of the regions $W_1, W_2, W_3, R_1$ and $R_2$ for the knot $T(3,n)$ with $n\equiv 4$ or $5$ (mod 6).}
  \label{T(3,n) grid 2b}
\end{figure}

\noindent  For $n\equiv 4 \text{ or 5 (mod 6)}$, $R_1\cap R_2 = W_1$, so the realizable pairs are exactly those in $W_1$ (excluding those pairs for which $e \not\equiv 2h $ (mod 4)).  For $n\equiv 1 \text{ or 2 (mod 6)}$, $R_1 =R_2$ so the remaining unknown pairs are in $U = W_1^c\cap R_1$.  The points in $U$ lie along a single line. For $n\equiv 1 \text{ (mod 6)}$,
$$U \subset \left\{(e,h): h = \frac{e}{2}+\frac{4n-4}{3}\right\}.$$ 
For $n\equiv 2 \text{ (mod 6)}$,
$$U \subset \left\{(e,h): h = \frac{e}{2}+\frac{4n-2}{3}\right\}.$$
\begin{figure}[ht]
  \centering
  \captionsetup{justification = centering}
  \includegraphics[width = 0.4\textwidth]{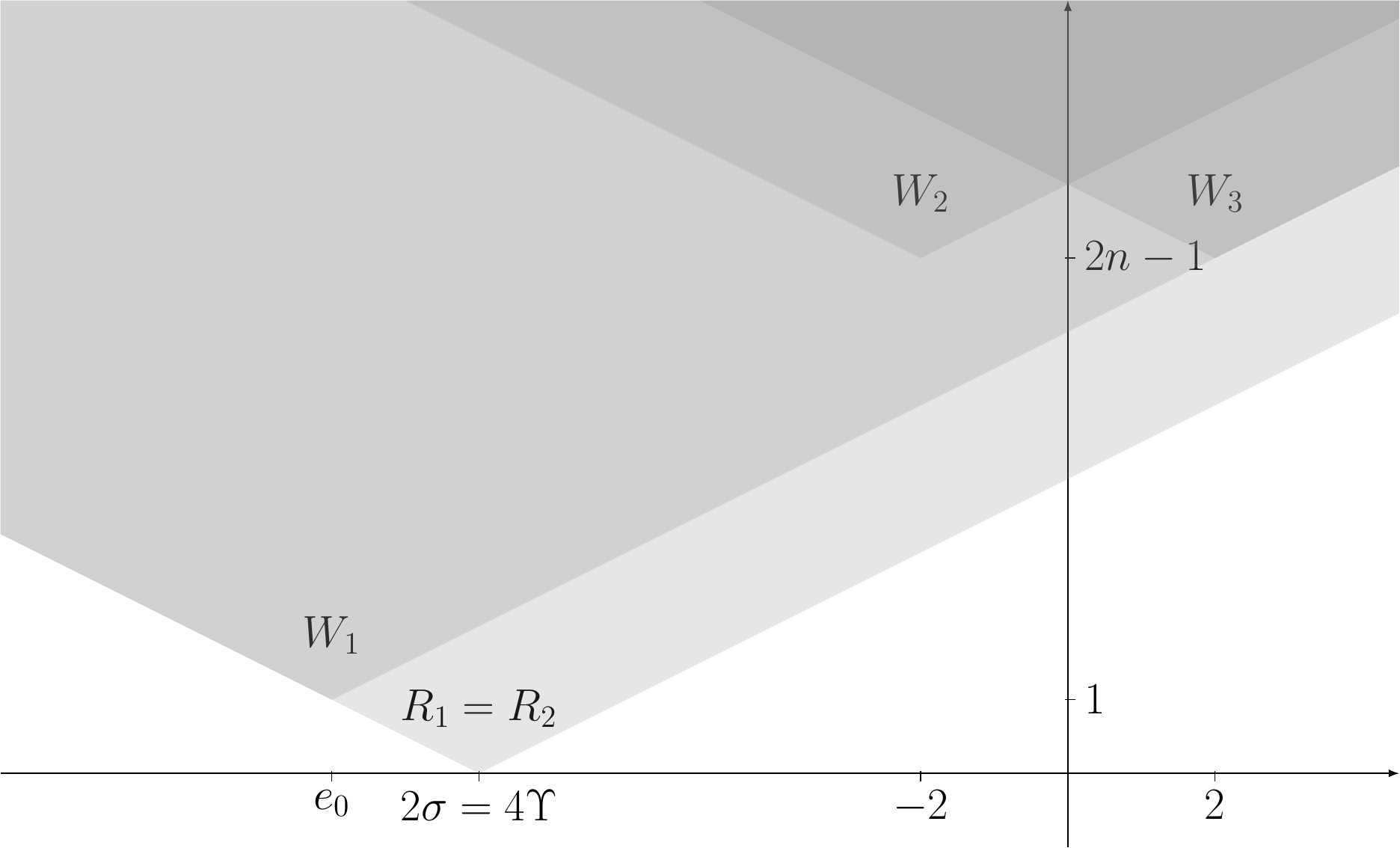}
  \caption{Schematic picture of the $(e,h)$--graph of the regions $W_1, W_2, W_3,$ and $R_1 = R_2$ for the knot $T(3,n)$ with $n\equiv 1$ or $2$ (mod 6).}
  \label{T(3,n) grid 2a}
\end{figure}

Finally, we attempt to rule out the punctured Klein bottle in each $U$ by computing:
\begin{center}
\begin{tabular}{c|c|c|c|c}
$n$ (mod 6) &$\sigma$ & Arf & $\sigma +4$Arf & (mod 8)\\
\hline
1&$\frac{-4n+4}{3}$&0&$\frac{-4n+4}{3}$&0\\
2&$\frac{-4n+2}{3}$&1&$\frac{-4n+2}{3}+4$&2\\
\end{tabular}
\end{center}
Thus Theorem \ref{klein obstruction} does not rule out points for $n\equiv 1 \text{ (mod 6)}$, but the theorem may be able to rule out points for $n\equiv 2 \text{ (mod 6)}$.   For $n\equiv 2 \text{ (mod 6)}$, there is one unknown point which corresponds to a punctured Klein bottle: $\left(\frac{-8n+4}{3}+4,2\right).$  Computations reveal that this point corresponds to a surface $F$ such that $\Sigma (F)$ is negative definite and so it (and, as a consequence, $(\frac{-8n+4}{3}+2,1)$) can be ruled out by Theorem \ref{klein obstruction}.  Thus the family of unknown points for $n\equiv 2 \text{ (mod 6)}$ is reduced by two, although still infinite.
\end{proof}


\section{Application of the $d$--invariant to the geography problem}

Let $K\subset S^3$ be a knot and $F \subset B^4$ be a nonorientable surface with $\partial F = K$.  Denote by $\Sigma (K)$ the two-fold branched cover of $S^3$ branched over $K$ and by $\Sigma (F)$ the two-fold branched cover of $B^4$ branched over $F$.  In the previous section, we noticed that all of the unknown points described in Theorems \ref{T(2,n)} and \ref{T(3,n)} correspond to surfaces $F$ with $\Sigma (F)$ negative definite.  This leads us to consider ruling out the existence of negative definite 4--manifolds with boundary $\Sigma (K)$, where $K$ is a knot.   Here, we use the $d$--invariant to do so for an infinite family of knots.  (We note that all torus knots have two--fold branched covers which bound negative definite manifolds (see \cite{neumann-raymond}), so we cannot apply the following strategy directly to individual torus knots.)  

In \cite{os1}, Ozsv\'ath and Szab\'o introduced the $d$--invariant.  The $d$--invariant associates to a rational homology sphere $Y$ with Spin$^\text{c}$--structure $\t$, a rational number denoted $d(Y,\t)$. Here we list some useful facts of the $d$--invariant.

\begin{thm}[[Ozsv\'{a}th and Szab\'{o}, \cite{os1}]
$ $
\begin{itemize}
\item If $Y$ is a closed, oriented 3--manifold and $\s \in$ Spin$^\text{c}(Y)$, then $d(Y, \s) = -d(-Y,\s)$.
\item If $(Y,\s)$ and $(Z, \t)$ are rational homology 3--spheres equipped with Spin$^\text{c}$--structures, then $d(Y\# Z, \s\#\t) = d(Y,\s) + d(Z, \t)$.
\end{itemize}
\end{thm}

Consider the following theorem.
\begin{thm}[Ozsv\'{a}th and Szab\'{o}, \cite{os1}] \label{os thm}
Let $Y$ be a rational homology 3--sphere and let $\t\in$ Spin$^\text{c}(Y)$.  Let $X$ be a smooth, negative definite 4--manifold with $\partial X = Y$ and let $\s\in$ Spin$^\text{c}(X)$ with $\s |_Y = \t$.  Then
$$c_1(\s)^2+b_2(X)\leq 4d(Y,\t).$$
\end{thm}
\noindent This has the following corollary, which we will apply to the geography problem.
\begin{cor}[Ozsv\'{a}th and Szab\'{o}, \cite{os1}] \label{os cor} If $Y$ is an integer homology 3--sphere with $d(Y) < 0$, then there is no negative definite 4--manifold X with $\partial X= Y$.
\label{OSS d-inv}
\end{cor}

Note that $\Sigma (K)$ is a rational homology sphere, so the $d$--invariant is defined for $\Sigma (K)$.  In \cite{manolescu-owens}, Manolescu and Owens introduce a knot invariant, $\delta$, defined by
$\delta (K) := 2 d(\Sigma (K), \t_0)$, where $\t_0$ is the Spin$^\text{c}$--structure induced by the unique Spin structure on $\Sigma (K)$.  They showed that $\delta$ is a concordance invariant which is additive under forming connected sums of knots and with the property that $\delta (-K) = - \delta (K)$.

 Rephrasing Corollary \ref{OSS d-inv} for the problem at hand:
\begin{cor} \label{OS cor}
Let $K\subset S^3$ be a knot such that $\Sigma (K)$ is a homology 3--sphere and $F$ a surface such that $\partial F = K$.  Suppose $\Sigma (F)$ has a negative definite intersection form.  Then $\delta(K)\geq0$.
\end{cor}
\noindent This yields the following bound on the nonorientable 4--genus for some classes of knots.
\begin{prop}\label{my bound}
Let $K\subset S^3$ be a knot.  If $K$ satisfies
\begin{itemize}
\item $\sigma(K) \leq 2\U_K(1)$,
\item  $\Sigma (K)$ is a homology 3--sphere, and
\item $\delta(K) < 0$,
\end{itemize}
then $$ \U_K(1)-\frac{\sigma (K)}{2} +1 \leq \gamma_4(K).$$
\end{prop}
\begin{proof}
Let $K\subset S^3$ be a knot which satisfies the assumptions of the proposition.  Combining Corollaries \ref{sigma wedge} and \ref{OSS cor} yields a wedge of pairs in the $(e,h)$ graph for $K$ corresponding to nonorientable surfaces in $B^4$ which may be realizable with boundary $K$.   Since $\sigma(K)\leq 2\U_K(1)$, this wedge includes a half-line of pairs which satisfy 
\begin{equation}\label{line} -\sigma(K)+\frac{e}{2} = h. \end{equation} 
 On the $(e,h)$ graph for $K$, these points are along the rightmost side of the wedge.  See Figure \ref{dThmSchematic} for a schematic picture.
\begin{figure}[ht]
  \centering
  \captionsetup{justification = centering}
  \includegraphics[width = 0.4\textwidth]{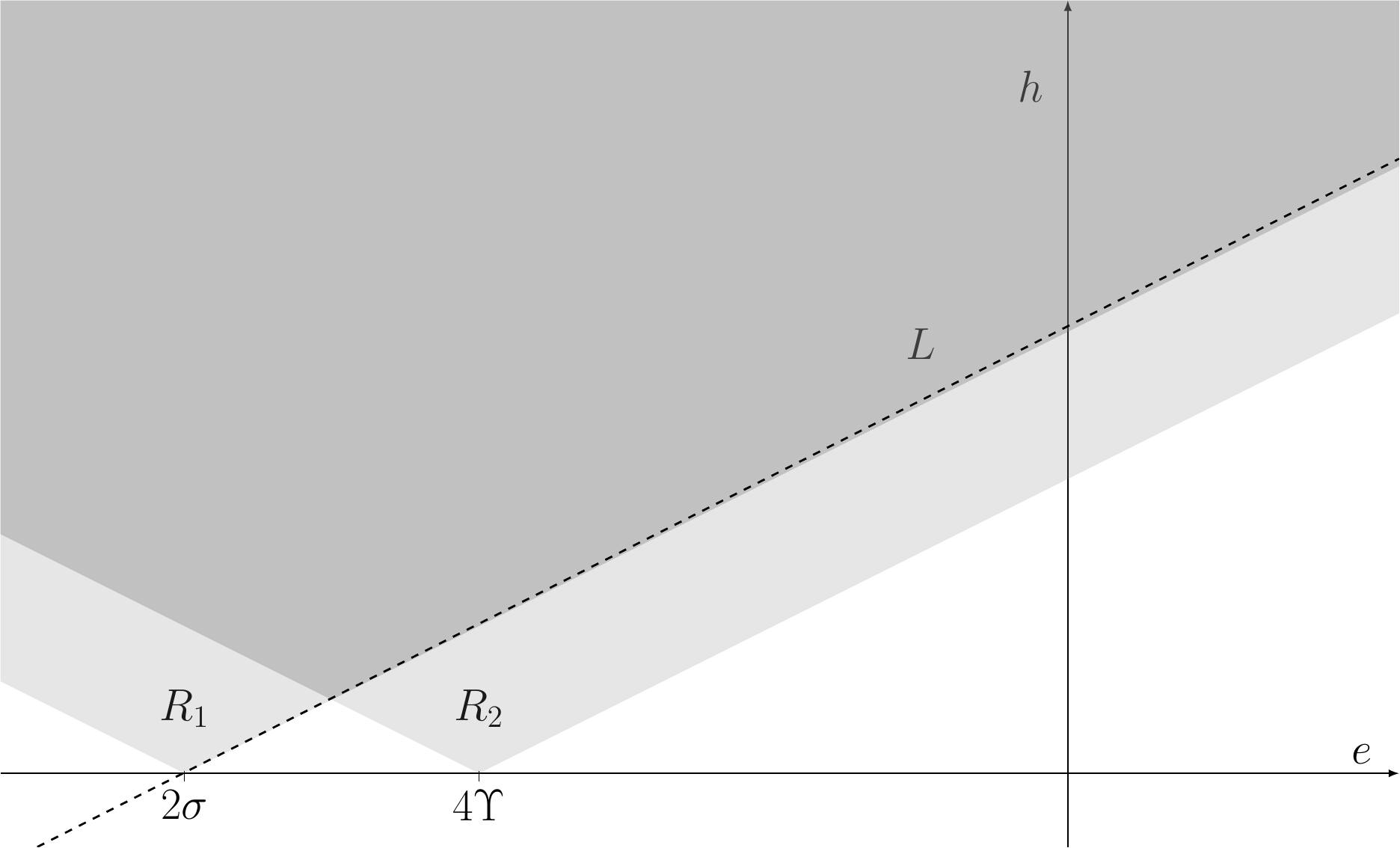}
  \caption{Schematic picture of the $(e,h)$--graph in the situation of Proposition \ref{my bound}.  The regions $R_1$ and $R_2$ are those arising from Corollaries \ref{sigma wedge} and \ref{OSS cor}, respectively. The line $L$ is that described in Equation \eqref{line}.}
  \label{dThmSchematic}
\end{figure}

Suppose that one of the points lying along line (\ref{line}) is realizable by some nonorientable surface $F\subset B^4$ with $\partial F = K$. Theorem \ref{sign(W(F))} implies that $h(F) = -\text{sign}(\Sigma(F))$.  Applying Lemma \ref{massey lem}, we have that
$$-\text{sign}(\Sigma(F)) = h(F) = b_1(F) = b_2(\Sigma(F)).$$  Thus $\Sigma (F)$ has a negative definite intersection form.  Since $\Sigma (K)$ is a homology 3--sphere, we can invoke Corollary \ref{OS cor}.  So it must be that $\delta(K) \geq 0$, a contradiction to our assumption that $\delta (K)$ is negative.  Thus no point along the line (\ref{line}) can be realized.  Note that the point $(2\U_K(1) + \sigma(K),\U_K(1)-\frac{\sigma(K)}{2})$ was the minimum point for the original wedge and it satisfies Equation \ref{line}.  Thus the lower bound in Corollary \ref{OSS bound} can be increased by 1, as desired.
\end{proof}

We now seek to find examples of such knots. Given a torus knot $T(p,q)$, Milnor \cite{milnor2} showed that its two-fold branched cover $\Sigma (T(p,q))$ is the Brieskorn manifold $\Sigma (2,p,q)$.  If $p$ and $q$ are relatively prime and odd, then $\Sigma (2,p,q)$ is an integer homology sphere and so has a unique Spin$^\text{c}$--structure.  We will omit the Spin$^\text{c}$--structure in the notation for the $d$--invariant. In \cite{os1}, Ozsv\'{a}th and Szab\'{o} computed $d(-\Sigma (2, 3,6n\pm  1))$, and, in \cite{tweedy}, Tweedy computed $d(-\Sigma (2, 5,n))$ for $(2,5,n)$ relatively prime and $d(-\Sigma (2, 7,n))$ for
$(2, 7,n)$ relatively prime. In \cite{nemethi}, Nemethi gave an algorithm for computing $d(\Sigma(p,q,r))$ for $p$, $q$, and $r$ relatively prime.  See also \cite{borodzik-nemethi, can-karakurt}.  

Thus we can compute $\delta(T(p,q))$ with $p$ and $q$ odd. Since the connected sum of two integer homology spheres is again an integer homology sphere, and since all of the relevant invariants are additive under forming connected sums, we will consider connected sums of torus knots and construct an infinite family of knots satisfying the conditions of Proposition \ref{my bound}.  In this way we can rule out infinitely many $(e,h)$--pairs which were previously unknown for the given knot.

\begin{thm}\label{connected sum d-inv} 
Let $c \geq 1$ and $K= cT (5,9)\#-(c +1)T (5,13)$ be the connected sum of $c$ copies of $T (5,9)$ and $(c +1)$ copies of $-T (5,13)$. Then $\Upsilon_K(1) - \frac{\sigma (K)}{2} = c -1$ and $\delta(K)=-4$.
\end{thm}
\begin{proof}
First, we compute that
$$\frac{\sigma(T(5,9))}{2}=-12 \text{ and } \frac{\sigma (T(5,13))}{2}=-16,$$
$$\Upsilon_{T(5,9)}(1)=-10 \text{ and } \Upsilon_{T(5,13)}(1)=-15,$$
$$\delta(T(5,9))=4 \text{ and } \delta(T(5,13))=4.$$
Then,
$$\frac{\sigma (K)}{2} = -12c+16(c+1)=4c+16,$$
$$\Upsilon_K(1)=-10c+15(c+1)=5c+15,$$
$$\delta(K)=4c-4(c+1)=-4.$$
Thus
$$\Upsilon_K(1)-\frac{\sigma (K)}{2} = 5c+15-(4c+16)=c-1.$$
\end{proof}
\begin{figure}[ht]
  \centering
  \captionsetup{justification = centering}
  \includegraphics[width = 0.4\textwidth]{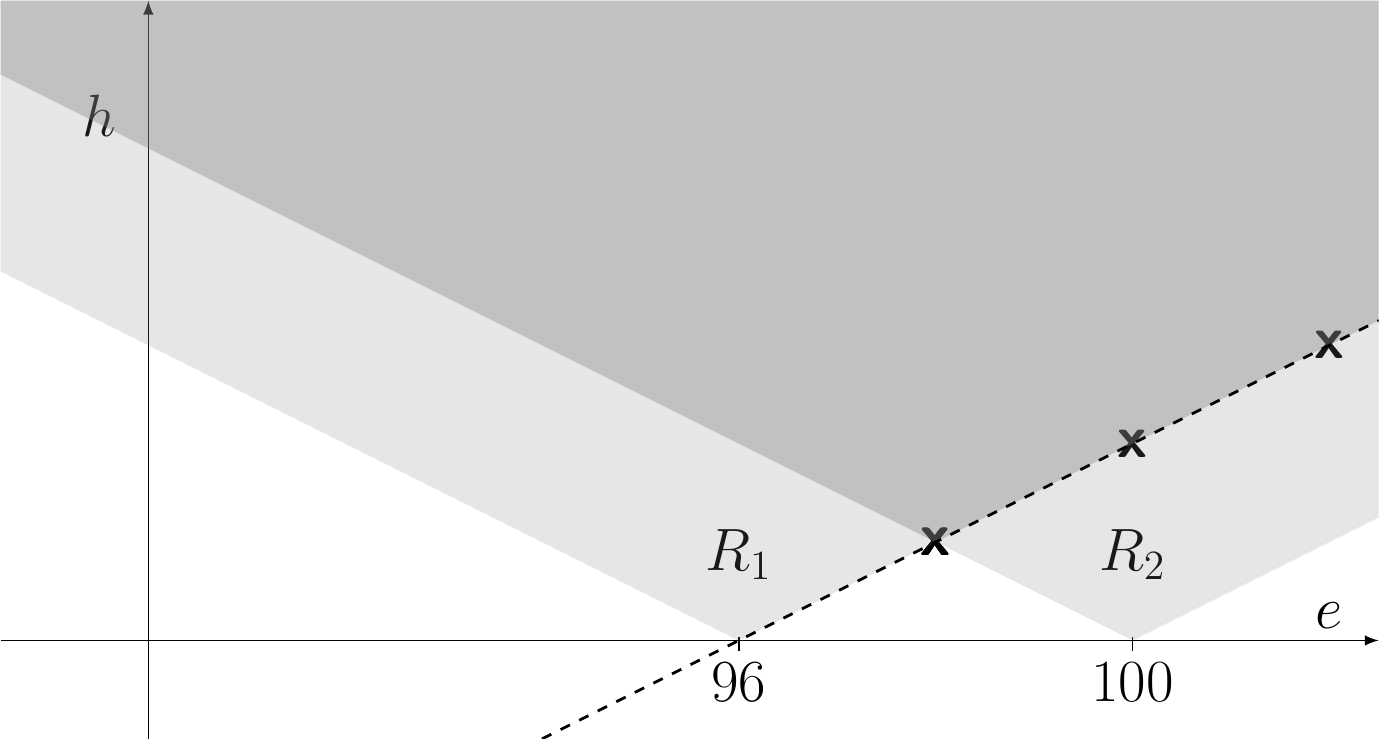}
  \caption{Schematic picture of the $(e,h)$--graph for the knot $2T(5,9)\#-3T(5,13)$.  The regions $R_1$ and $R_2$ are those arising from Corollaries \ref{sigma wedge} and \ref{OSS cor}, respectively. Points marked with an $\sf{x}$ have been ruled out by Proposition \ref{my bound}, increasing the lower bound on $\gamma_4$ from $1$ to $2$.}
  \label{T(5,9)-T(5,13)}
\end{figure}
\begin{cor}
 Let $c \geq1$ and $K= cT (5,9)\#-(c +1)T (5,13)$ be the connected sum of $c$ copies of
$T (5,9)$ and $(c +1)$ copies of $-T (5,13)$. Then $c \leq \gamma_4(K) \leq 3c +1$.
\label{connected sum h}
\end{cor}
\begin{proof}
The leftmost inequality is a consequence of Proposition \ref{my bound}.  The upper bound follows from the fact that $\gamma_4(T (5, 9)) \leq 2$ and $\gamma_4(T (5,13)) = 1$.  Performing a band move on each knot reveals upper bounds of 2 and 1 respectively.\end{proof}


\section{Further remarks}
We have seen that for some torus knots, for example $T(2,3)$ and $T(3,4)$, the geography problem is completely solved.  However, for many small knots there are still several unknown values.

In cases where the two-fold branched cover of $S^3$ branched over the knot is a manifold which is well understood, it is possible that more can be said about the geography problem, such as in the case of two-bridge knots. 

  While much is known about rational homology spheres bounding definite manifolds (see, for instance, \cite{choe-park, owens-strle1, owens-strle2, os1}), the cases of semi-definite and indefinite manifolds are more challenging.
\begin{exmp} Consider the figure-eight knot, $4_1$.  Because $4_1$ is an amphicheiral knot, $\sigma(4_1) = \U_{4_1}(1) = 0$.  Viro \cite{viro} showed that $4_1$ does not bound a M{\"o}bius band in the 4--ball.  In Figure \ref{2kand41}(b), we see that $4_1$ does bound a punctured Klein bottle $F$ and computation reveals that $e(F)=4$.  Taking the mirror image, we see that $4_1$ also bounds a punctured Klein bottle with $e = -4$.  This results in an $(e,h)$--graph with exactly one unknown point: $(0,2)$.
\begin{question}
Does the figure-eight knot bound a punctured Klein bottle $F$ with $e(F) = 0$?
\end{question}
\noindent Note that Theorem \ref{sign(W(F))} implies that, for such an $F$,  $\text{sign}(\Sigma(F)) = 0$ and so $\Sigma(F)$ is not a definite manifold.
\end{exmp}


\bibliography{mybiblio}
\bibliographystyle{abbrv} 

\end{document}